\documentclass{amsart}
\usepackage{amsfonts}
\usepackage{amsmath}
\usepackage{amssymb}
\usepackage{dsfont}
\usepackage{amscd}
\usepackage{amsthm}

\newtheorem{theorem}{Theorem }

\theoremstyle{definition}

\theoremstyle{remark}

\def\iy{\infty}
\def\hf{{1\over 2}}

\def\be{\begin{equation}}
\def\ee{\end{equation}}
\def\ba{\begin{eqnarray*}}
\def\ea{\end{eqnarray*}}
\def\bae{\begin{eqnarray}}
\def\eae{\end{eqnarray}}
\def\bc{\begin{center}}
\def\ec{\end{center}}

\begin{document}
\title[Explicit Probability Densities Associated with Fuss-Catalan
Numbers] {On Explicit Probability Densities Associated with
Fuss-Catalan Numbers}

\author[ D.-Z. Liu C. Song and Z.-D. Wang]{ Dang-Zheng Liu, Chunwei Song and Zheng-Dong Wang}

\address{School of Mathematical Sciences, LMAM, Peking University, Beijing,
100871, P.R. China} \email{dzliumath@gmail.com}
\email{csong@math.pku.edu.cn}
 \email{zdwang@pku.edu.cn}

\maketitle

\begin{abstract}
In this note we give explicitly a family of  probability densities,
the moments of which are Fuss-Catalan numbers. The densities appear
naturally in random matrices, free probability and other contexts.
\end{abstract}


\\

In this note we study  a family of probability densities $\pi_{s}$,
$s \in \mathds{N}$, which are uniquely determined by the moment
sequences $\{m_{0}, m_{1}, \ldots, m_{k}, \dots\}$ \cite{agt}. Here

\be m_{k}=\frac{1}{sk+1}\begin{pmatrix}
   sk+k  \\
  k
\end{pmatrix}\label{FCnumber}\ee
are known as Fuss-Catalan numbers in Free Probability Theory
\cite{ns}. The densities $\pi_{s}$ belong to the class of Free
Bessel Laws \cite{bbcc} and are known to appear in several different
contexts, for instance, random matrices \cite{agt,bbcc,ns}, random
quantum states \cite{cnz}, free probability and quantum groups
\cite{bbcc}. More precisely, $\pi_{s}$ is the limit spectral
distribution of random matrices in the forms like
$X^{s}_{1}X^{*s}_{1}$ and $X_{1}\cdots X_{s} X^{*}_{s} \cdots
X^{*}_{1}$, where $X_{1},\ldots, X_{s}$ are independent $N \times N$
random matrices; in free probability we have the free convolution
relation: $\pi_{s}=\pi_{1}^{\boxtimes s}$ \cite{ns,bbcc}.

More generally, T. Banica et al \cite{bbcc} introduce a remarkable
family of probability distributions $\pi_{s,t}$ with $(s,t)\in
(0,\iy)\times (0, \iy)-(0,1)\times (1, \iy)$, called free Bessel
laws. $\pi_{s}$ is the special case where $t=1$, i.e.,
$\pi_{s,1}=\pi_{s}$. The moments of $\pi_{s,t}$ are the Fuss-Catalan
polynomials (Theorem 5.2, \cite{bbcc}): \be
m_{k}(t)=\sum_{j=1}^{k}\frac{1}{j}\begin{pmatrix}
   k-1 \\
  j-1
\end{pmatrix}\begin{pmatrix}
   sk  \\
  j-1
\end{pmatrix}t^{j}.\ee
Indeed, the following relation holds \cite{bbcc}: \be
\pi_{s,1}=\pi_{1}^{\boxtimes s}, \ \ \ \pi_{1,t}=\pi_{1}^{\boxplus
t}.\ee

The distribution $\pi_{1,t}$, the famous Marchenko-Pastur law of
parameter $t$ or called free Poisson law,  owns an explicit formula:
\be
\pi_{1,t}=\max(1-t,0)\delta_{0}+\frac{\sqrt{4t-(x-1-t)^{2}}}{2\pi
x}.\ee In special case, \be
\pi_{1}=\pi_{1,1}=\frac{1}{2\pi}\sqrt{4x^{-1}-1}.\ee Another special
case of $\pi_{s,t}$ where an explicit density formula is available
is, due to Penson and Solomon \cite{ps}, \be
\pi_{2}=\pi_{2,1}=\frac{\sqrt[3]{2}\sqrt{3}}{12\pi}\frac{
\sqrt[3]{2}(27+3\sqrt{81-12x})^{2/3}-6\sqrt[3]{x}}{x^{2/3}(27+3\sqrt{81-12x})^{1/3}}\,1_{(0,27/4]}(x).\ee

To the best of our knowledge, except for the above special cases
there are no explicit formulae available for the other $\pi_{s,t}$.
The aim of this work is to give explicit densities of
$\pi_{s}=\pi_{s,1}$, $s \in \mathds{N}$. We state our main result as
follows and its proof is based on the method: how to find an
explicit density from a given certain moment sequence used in
\cite{wigner,lwy}.

\begin{theorem}
\label{theorem} Let $\pi_{s}$, $s \in \mathds{N}$ be the unique
densities determined by the Fuss-Catalan numbers in
Eq.(\ref{FCnumber}), then we have the following formulae
\begin{equation}
\pi_{s}(x)=\frac{1_{(0,K]}(x)}{B(\frac{1}{2},\frac{1}{2}+\frac{1}{s})}
\int_{[0,1]^{s}}\frac{\big(\tau K-x\big)^{1/s-1/2}}{\sqrt{x}
\,\big(\tau K\big)^{1/s}} F(t_{1},\ldots,t_{s})\,1_{\{\tau K\geq
x\}}d^{s} t,
\end{equation}
where $K=(s+1)^{s+1}/s^{s}$,
$\tau=\displaystyle\prod_{j=1}^{s}t_{j}$ and
$F(t_{1})=\delta(t_{1}-1)$ while for $s>1$ \begin{align}
F(t_{1},\ldots,t_{s})&=
\frac{1}{B(\frac{1}{s+1},\frac{s-1}{2s+2})\displaystyle\prod_{j=2}^{s}B(\frac{j}{s+1},\frac{j}{s(s+1)})} \times\nonumber\\
&
t^{\frac{1}{s+1}-1}_{1}(1-t_{1})^{\frac{1}{2s+2}-1}\displaystyle\prod_{j=2}^{s}
t^{\frac{j}{s+1}-1}_{j}(1-t_{j})^{\frac{j}{s(s+1)}-1}.
\end{align}

\end{theorem}

\begin{proof}
First, we derive a family of symmetric distributions $\sigma_{s}$,
the $2k$-moments of which are $m_{k}$ in Eq.(\ref{FCnumber}).

Consider the characteristic function of $\sigma_{s}$ as follows:
\begin{align}
  &\int_{-\iy}^{ +\iy} e^{i \xi x} \sigma_{s}(x) d x
 =
 \sum\limits_{k=0}^{\infty}\frac{(-\xi^{2})^{k}}{(2k)!}m_{k}
 =\sum\limits_{k=0}^{\infty}\beta_{k}\frac{(-\xi^{2})^{k}}{k!}
 ,\end{align}
where \be\beta_{k}=\frac{1}{sk+1}\begin{pmatrix}
   sk+k  \\
  k
\end{pmatrix}\frac{k!}{(2k)!}=\frac{1}{sk+1}
   \frac{(sk+k)!}{(sk)!(2k)!}.\ee

   A direct computation shows that the ratio
\begin{align}\frac{\beta_{k+1}}{\beta_{k}}&=\frac{sk+1}{s(k+1)+1}\frac{1}{(2k+1)(2k+2)}
\frac{\big((s+1)k+1\big)\big((s+1)k+2\big)\cdots
\big((s+1)k+s+1\big)}{(sk+1)(sk+2)\cdots
(sk+s)}\nonumber\\
&=\frac{K}{4} \frac{(k+\frac{1}{s+1})(k+\frac{2}{s+1})\cdots
(k+\frac{s}{s+1})}{(k+\frac{1}{2})(k+\frac{2}{s})\cdots
(k+\frac{s}{s})}\frac{1}{k+1+\frac{1}{s}}\nonumber\\
&\doteq \frac{K}{4}\frac{(k+a_{1})(k+a_{2})\cdots
(k+a_{s})}{(k+b_{1})(k+b_{2})\cdots
(k+b_{s})}\frac{1}{k+b_{s+1}}\nonumber.\end{align} Here \be
b_{1}=\hf,\  a_{i}=\frac{i}{s+1} \ \textrm{and}\
b_{i+1}=\frac{i+1}{s} \  \textrm{for}\ i=1, 2, \dots, s.\ee

Therefore, using the generalized hypergeometric function, we rewrite
\begin{align}\int_{-\iy}^{ +\iy}  e^{i \xi x} \sigma_{s}(x) d x&= \,_{s}F_{s+1}(a_{1},\dots,a_{s};b_{1}, \dots, b_{s},b_{s+1};
-\frac{K}{4}\xi^{2})\\
&= \label{ftransform} \int_{[0,1]^{s}}F(t_{1},\ldots,t_{s})
\,_{0}F_{1}(b_{s+1}; -\frac{\tau K}{4}\xi^{2})\,d^{s} t.
\end{align}
Note that $a_{1}=b_{1}$ for $s=1$ but $b_{j}>a_{j}>0, j=1, 2,
\ldots, s$ when $s>1$, in Eq.(\ref{ftransform}) we have made use of
Euler's integral representation of the generalized hypergeometric
function \cite{aar}.

Next, with the help of the integral representation of Bessel
function of the first kind \cite{aar}, that is, for $\alpha>-\hf$,
 \begin{align}
  J_{\alpha}(z)&=\frac{(z/2)^{\alpha}}{\Gamma(\alpha+1)}\,_{0}F_{1}(\alpha+1;
-\frac{1}{4}z^{2})\nonumber\\
&=\frac{(z/2)^{\alpha}}{\sqrt{\pi}\Gamma(\alpha+\frac{1}{2})}
\int_{-1}^{1}\ e^{i z x}(1-x^{2})^{\alpha-\frac{1}{2}}d x,
\end{align}
we get
 \begin{align}
 \,_{0}F_{1}(\alpha+1;
-\frac{1}{4}z^{2})&=\frac{\Gamma(\alpha+1)}{\sqrt{\pi}\Gamma(\alpha+\frac{1}{2})}
\int_{-1}^{1}\ e^{i z x}(1-x^{2})^{\alpha-\frac{1}{2}}d x\nonumber\\
&=\frac{1}{B(\hf,\alpha+\hf)}\int_{-1}^{1}\ e^{i z
x}(1-x^{2})^{\alpha-\frac{1}{2}}d x.
\end{align}

Set $\alpha=1/s$, $z=\sqrt{\tau K}\xi$, we have
\begin{align}\label{integralrep}
 \,_{0}F_{1}(b_{s+1};
-\frac{\tau
K}{4}\xi^{2})&=\frac{1}{B(\hf,\hf+\frac{1}{s})}\int_{-1}^{1}\ e^{i
\sqrt{\tau K}\xi x}(1-x^{2})^{\frac{1}{s}-\frac{1}{2}}d x \nonumber\\
&=\frac{1}{B(\hf,\hf+\frac{1}{s})}\int_{-\sqrt{\tau K}}^{\sqrt{\tau
K}} e^{i \xi x} \frac{ (\tau K-x^{2})^{\frac{1}{s}-\frac{1}{2}} }
{(\tau K)^{\frac{1}{s}}} d x .
\end{align}

Combining  (\ref{integralrep}) and (\ref{ftransform}), after
interchanging the order of integration,  we then obtain
\begin{equation}
\sigma_{s}(x)=\frac{1}{B(\frac{1}{2},\frac{1}{2}+\frac{1}{s})}
\int_{[0,1]^{s}}\frac{\big(\tau K-x^{2}\big)^{1/s-1/2}}{\big(\tau
K\big)^{1/s}} F(t_{1},\ldots,t_{s})\,1_{\{\tau K\geq x^{2}\}}d^{s}
t.
\end{equation}

Notice the fact
 \be \pi_{s}(x)=\frac{\sigma_{s}(\sqrt{x})}{\sqrt{x}}1_{(0, \iy)},\ee
 the proof is then complete.
\end{proof}

At the end we remark that some properties of the density
$\pi_{s}(x)$ follow easily from Theorem \ref{theorem}, for instance,
there are no atoms; the support is $[0,K]$; the density is analytic
on $(0,K)$ (Theorem 2.1, \cite{bbcc}).


\end{document}